\documentclass[11pt]{amsart}
\usepackage{amssymb,color,cite}
\usepackage{graphicx, tikz-cd}

\DeclareFontFamily{U}{mathx}{}
\DeclareFontShape{U}{mathx}{m}{n}{<-> mathx10}{}
\DeclareSymbolFont{mathx}{U}{mathx}{m}{n}
\DeclareMathAccent{\widehat}{0}{mathx}{"70}
\DeclareMathAccent{\widecheck}{0}{mathx}{"71}

\newtheorem{theorem}{Theorem}[section]

\newtheorem{corollary}[theorem]{Corollary} 
\newtheorem{proposition}[theorem]{Proposition}

\theoremstyle{definition}

\newtheorem{example}[theorem]{Example}

\theoremstyle{remark}
\newtheorem{remark}[theorem]{Remark}

\numberwithin{equation}{section}
\usepackage{hyperref}
\newcommand{\be}[1]{\begin{equation}\label{#1}}
\newcommand{\ee}{\end{equation}}

\DeclareMathOperator{\WF}{WF}

\allowdisplaybreaks
\newcommand{\Rn}{\mathbb{R}^n}
\newcommand{\rb}{\mathbb{R}}
\newcommand{\Sc}{\mathcal{S}}
\renewcommand{\d}{\,\mathrm{d}}
\renewcommand{\sb}{\mathbb{S}}
\newcommand{\Fc}{\mathcal{F}}
\newcommand{\I}{\mathrm{i}}

\DeclareMathOperator{\supp}{supp}
\newcommand{\PDO}{{\rm $\Psi$DO}}

\usepackage{mathtools}
\mathtoolsset{showonlyrefs}

\newcommand{\note}[1]{\!\!}

\begin{document}

\title{The Light ray transform for pseudo-Euclidean metrics}

\author{Divyansh Agrawal}
\address{Centre for Applicable Mathematics, Tata Institute of Fundamental Research, Bangalore, India and Department of Mathematics, Purdue University, West Lafayette, IN 47907}

\email{agrawald@tifrbng.res.in, agrdiv01@gmail.com}
\thanks{D.A. did this work as a SERB's Overseas Visiting Doctoral Fellow (OVDF) at Purdue University.}

\author{Plamen Stefanov}

\address{Department of Mathematics, Purdue University, West Lafayette, IN 47907}
\email{stefanov@math.purdue.edu}
\thanks{P.S.~partly supported by  NSF  Grant DMS-2154489.}

\date{}

\dedicatory{}

\begin{abstract}
	We study the ray transform $L$ over null (light) rays in the  pseudo-Euclidean space with signature $(n',n'')$, $n'\ge2$, $n''\ge2$. We analyze the normal operator $L'L$, derive an inversion formula, and prove stability estimates. We show that the symbol $p(\xi)$ is elliptic but singular at the light cone $\mathcal L$ with conormal singularities there.  We analyze  $L$ as a Fourier Integral Operator as well. Finally, we compare this to the Minkowski case.
\end{abstract}

\maketitle

\section{Introduction}
The purpose of this article is to study the light ray transform in the  pseudo-Euclidean space. We consider  $\rb^{n}$, where $n=n'+n''$, with $n',n'' \geq 2$. Throughout this article, we  identify $\Rn$ with $\rb^{n'} \times \rb^{n''}$, and denote the points in $\rb^{n}$ by $x = (x',x'') \in \rb^{n'} \times \rb^{n''}$, similarly for vectors and covectors. We assume that the metric $g$ is given by
$$
g = \mathrm{diag}(\underbrace{-1, \dots, -1}_{n'}, \underbrace{1, \dots, 1}_{n''}),
$$
and, in particular,   has signature $(n',n'')$. 
The case $n'=1$ (or $n''=1$ but not $n'=n''=1$ which is not interesting) gives us the Minkowski metric, and the light ray transform for it has been studied already even for more general Lorentzian metrics, see, e.g., \cite{MR1004174, Lauri-Light-21,LOSU-Light_Ray, S-support2014, Vasy-Wang-21, wang2017parametrices, SU-book, Yiran-Light-2021}.  
Consider the operator
\begin{equation}
Lf (\gamma) = \int_{\gamma} f \d s,
\end{equation}
which integrates $f$ along null-lines (also called light line/ray) with respect to $g$, satisfying  $g(\dot{\gamma}, \dot{\gamma}) = 0$, $\dot\gamma\not=0$, and, say, $f\in C_0^\infty(\rb^n)$.  Recall that any line in $\rb^{n}$ can be parameterized by 
\begin{equation}
t \mapsto x + t \theta, \quad \text{for } x, \theta \in \Rn.
\end{equation}
For such a line to be a null line, we require $g(\theta, \theta) = 0$, $\theta\not=0$, or $|\theta'| = |\theta''|\not=0$. Thus, we have the following parameterization for the light ray transform:
\begin{equation}\label{def-in}
Lf(x,\theta) = \int_{\rb} f(x+t\theta) \d t \quad \text{~ for ~} x,\theta \in \Rn, \text{~with~} |\theta'| = |\theta''|\not=0,
\end{equation}
and we require below $x\cdot\theta=0$  and  $|\theta'| = |\theta''|=1$. 
Thus, the transform studied here is the restriction of the usual ray transform to null-lines with respect to $g$, and a generalization of the Light ray transform. The problem of recovering $f$ from the knowledge of $Lf$ is formally overdetermined by  $(2n-3)-n=n-3$ variables since $Lf$ depends on $n-1+n-2$ variables, see \eqref{Sigma} below, and $f$ depends on $n$ variables.

The transform $L$ appears naturally as the bicharacteristic one, restricted to functions $f(x)$ on the base, related to the generalized ultra-hyperbolic operator $P=\Delta_{x'}-\Delta_{x''}$, see \cite{OSSU-principal}. In particular, its inversion would recover a compactly supported potential $V$ from Cauchy data for $P+V$ on a boundary of a smooth domain containing $\supp V$ by \cite[Theorem~1.2]{OSSU-principal}.  
Also, $L$ belongs to the class of the generalized Radon transforms, formulated first by Guillemin \cite{Guillemin85,GuilleminS}. Restricted versions of them to $n$-dimensional complexes of curves has been studied microlocally in \cite{Greenleaf_Uhlmann90,Greenleaf_UhlmannCM,Greenleaf-Uhlmann}. Its analytic version has been studied recently in \cite{mazzucchelli2023Analytic} establishing support theorems. The same transform has been studied in \cite{Ilmavirta-pseudo}, where it is shown that $L$ is injective on $\mathcal{S}(\rb^n)$ using the Fourier Slice Theorem. 

In sections~\ref{sec_3}, \ref{sec_4}, we analyze the normal operator $L'L$ and show that it is a convolution, hence a Fourier multiplier with some $p(\xi)$ which is positively homogeneous of order $-1$, continuous but singular at the light cone $\mathcal{L}$ in $T^*\rb^n$ with conormal singularities but still elliptic when either $n'\ge3$ or $n''\ge3$. In section~\ref{sec_co} we  compute the principal symbol of $p$ as a conormal distribution at $\mathcal{L}$.  When $n'=n''=2$, $p(\xi)$ has a logarithmic blow-up at   $\mathcal L$; still elliptic in the sense that it has a lower bound $C/|\xi|$. In both cases, we present an inversion formula and prove a stability estimate which requires a modified Sobolev space when $n'=n''=2$.  The symbol $p$ does not seem to have a closed-form formula, except in some particular dimensions (see section~\ref{example}), and we present several representations. 

In section~\ref{sec_FIO}, we analyze $L$ briefly as a Fourier Integral Operator (FIO) to fit it into the generalized Radon transform framework of Guillemin. We do not get deeply into this analysis since the singularity of $p(\xi)$ at the light cone suggests the need to use a specialized calculus which is beyond the scope of this work. We refer to the remarks following Theorem~\ref{thm_FIO}. Finally, in section~\ref{minkowski}, we give a brief comparison to the Minkowski case.

\textbf{Acknowledgments:} The authors thank Adri Olde Daalhuis for various discussions regarding Hypergeometric functions. 

\section{Setup and uniqueness}

\subsection{Setup} 
We call the cone 
\[
\mathcal{L}:= \{(\theta',\theta'')\in \rb^n|\; |\theta'|=|\theta''|\not=0\}
\]
the \textit{light cone} (in $T\rb^n$); and we use the same notation  for the cone   $|\xi'|=|\xi''|\not=0$ in the cotangent bundle since $g^{-1}$ is the co-metric formally looking the same as $g$. If we start with $\theta$ not restricted to be of fixed length, and  non-restricted $x$ in \eqref{def-in}, we would have, for $a \in \rb \setminus \{0\}$,  
\begin{equation}
Lf(x, a \theta) = \frac{1}{|a|} Lf(x, \theta).
\end{equation}
Due to this scaling property and the fact that $|\theta'| = |\theta''|$, we impose the restriction $|\theta'| = |\theta''| =1$. Thus, we now have 
\begin{equation}   \label{theta}
\theta = (\theta', \theta'') \in \sb^{n'-1} \times \sb^{n''-1}.
\end{equation}
We trivially have, analogous to the corresponding property of the X-ray transform, for $x, \theta \in \Rn$ and $s \in \rb$
\begin{equation}
Lf(x+s\theta, \theta) = Lf(x,\theta).
\end{equation}
We take $x\perp\theta$. Our operator now is defined as
\begin{equation}\label{def}
    Lf(x,\theta) = \int_{\rb} f(x+t\theta) \d t,
\end{equation}
where $\theta = (\theta',\theta'') \in \sb^{n'-1} \times \sb^{n''-1}$, $x= (x',x'') \in \rb^{n'} \times \rb^{n''}$ with $x\cdot\theta=0$. The variety of those $(x,\xi)$ will be called $\Sigma$, i.e.,  
\begin{equation}   \label{Sigma}
\Sigma := \left\{ (x,\theta)|\; x\in \Rn,\; \theta\in \sb^{n'-1}\times \sb^{n''-1},\; x\cdot\theta=0 \right\}.
\end{equation}
One can think of $\Sigma$ as the normal bundle of $\sb^{n'-1}\times \sb^{n''-1}$, with the standard smooth structure on it. We equip it with the measure $\d  H_\theta(x) \d S(\theta')\d S(\theta'')$, where $\d H_\theta(x)$ is the standard Lebesgue  measure on the plane  $\theta^\perp$  and $\d S(\cdot)$ denotes the surface measure on the unit spheres in the corresponding variable. 

\subsection{Fourier Slice Theorem} We now proceed to derive the Fourier Slice theorem for the operator \eqref{def}. To fix our notation, we recall the Fourier transform. For $f$  in the Schwartz class $\Sc(\Rn)$, the Fourier transform is defined as 
\begin{equation}
    \Fc f (\xi) = \widehat{f}(\xi) = \int_{\Rn} e^{-\I x \cdot \xi} f(x) \d x,
\end{equation}
and the inverse Fourier transform is given as
\begin{equation}
    \Fc^{-1} f (x) = \widecheck{f} (x) = (2\pi)^{-n} \int_{\Rn} e^{\I x \cdot \xi} f(\xi) \d \xi.
\end{equation}
We similarly define the Fourier and inverse Fourier transforms on a hyperplane through the origin, say $\theta^\perp$, denoted $\Fc_{\theta^\perp}$ and $\Fc^{-1}_{\theta^\perp}$ respectively. More precisely, for $x,\xi \in \theta^\perp$,
\begin{equation}
    \begin{split}
        \Fc_{\theta^\perp} f (\xi) &= \int_{\theta^\perp} e^{-\I x \cdot \xi} f(x) \d H_\theta(x), \\
        \Fc^{-1}_{\theta^\perp} f (x) &= (2\pi)^{1-n} \int_{\theta^\perp} e^{\I x \cdot \xi} f(\xi) \d H_\theta(\xi),
    \end{split}
\end{equation}
where $\d H_\theta$ denotes the natural $n-1$ dimensional (Lebesgue) measure on $\theta^\perp$. 

\begin{theorem}[Fourier Slice Theorem]   \label{thm_FS}
    For every $f \in L^1 (\Rn)$, $n'\ge2$, $n''\ge2$,
    \begin{equation}
       \hat f(\xi) = \int_{\theta^\perp } e^{-\I x\cdot\xi} Lf(x,\theta)\,\d H_\theta(x), \quad \text{whenever  $\xi\perp \theta$, $\theta\in  \sb^{n'-1}\times \sb^{n''-1} $}.
    \end{equation}
\end{theorem}
The theorem is just a restricted version of the classical Fourier Slice Theorem.

\subsection{Consequences of the Fourier Slice Theorem} We immediately obtain the following implications of the Fourier Slice Theorem, see also \cite{Ilmavirta-pseudo}.

\begin{corollary}[Uniqueness] \label{cor_uniq}
	$f\in L^1(\Rn)$ and $Lf=0$ imply $f=0$.
	\end{corollary}

\begin{proof} 
	We follow \cite{Ilmavirta-pseudo} here. 
Given $\xi\in\Rn$, we want to show that the equation $\xi'\cdot\theta' + \xi''\cdot\theta''=0$ has a solution  $\theta\in \sb^{n'-1}\times \sb^{n''-1}$. Assume, without loss of generality, that $|\xi'|\le |\xi''|$. Choose any $\theta'\in  \sb^{n'-1}$. Then $|\xi'\cdot\theta' |\le |\xi'| \le |\xi''|$, hence we can find $\theta''\in \sb^{n''-1}$ so that $ \xi''\cdot\theta''=- \xi'\cdot\theta'$.

By Theorem~\ref{thm_FS}, $\hat f(\xi)=0$ for every $\xi$, which completes the proof. 
\end{proof} 	

\section{The normal operator} \label{sec_3}
\subsection{The adjoint and the normal operators} The operator $L:C_0^\infty(\rb^n)\to C_0(\Sigma)$ is continuous, see \cite[sec.~2.8]{Friedlander1998}. Then it has a sequentially continuous transpose $L':\mathcal{D}'(\Sigma) \to \mathcal{D}'(\rb^n)$. We compute $L'$ below on the dense set $C_0^\infty(\Sigma)$. 

\begin{theorem}\label{th:normal}
	Let $n', n'' \geq 2$. The transpose of $L$ is given by 
	\begin{equation}\label{transpose}
		L' \phi (z) = \int_{\sb^{n'-1}} \int_{\sb^{n''-1}} \phi(z- {2^{-1}}(z \cdot \theta) \theta, \theta) \d S(\theta'') \d S(\theta'), \quad \forall \phi\in C_0^\infty(\Sigma). 
	\end{equation}
The normal operator is
	\begin{align}\label{normal}
		L'Lf(x) &= 2 \int_{\sb^{n'-1}} \int_{\sb^{n''-1}} \int_{0}^{\infty} f\left ( x + t \theta \right) \d  t \d S(\theta'') \d S(\theta'), \quad \forall f\in C_0^\infty(\rb^n). 
	\end{align}
\end{theorem}
\begin{proof}
	 For every $\phi\in C_0^\infty(\Sigma)$, we have 
\begin{equation}
    \begin{split}
        \langle Lf, \phi \rangle &= \int_{\sb^{n'-1}} \int_{\sb^{n''-1}}   \int_{\theta^\perp} Lf (x, \theta) \phi (x, \theta) \d H_\theta(x)  \d S(\theta'') \d S(\theta')\\
        &= \int_{\sb^{n'-1}} \int_{\sb^{n''-1}}   \int_{\theta^\perp} \int_{\rb} f (x+ t \theta) \phi (x, \theta) \d t \d H_\theta(x) \d S(\theta'') \d S(\theta') \\
        &= \int_{\sb^{n'-1}} \int_{\sb^{n''-1}}   \int_{\Rn} f (z) \phi (z- {2^{-1}}(z\cdot\theta)\theta, \theta)\d z \d S(\theta'') \d S(\theta')   .    
    \end{split}
\end{equation}
This proves \eqref{transpose}.  To prove \eqref{normal}, write 
\begin{align}
	L'Lf(x) &= \int_{\sb^{n'-1}} \int_{\sb^{n''-1}} \int_{\rb} f\left ( x - {2^{-1}}(x\cdot  \theta )  \theta + t \theta \right) \d  t \d S(\theta'') \d S(\theta') \\
	&= \int_{\sb^{n'-1}} \int_{\sb^{n''-1}} \int_{\rb} f\left ( x + t \theta \right) \d  t \d S(\theta'') \d S(\theta').  \label{L'Lf}
\end{align}
Let us split the integral in $t$ variable into $\{t < 0\}$ and $\{t> 0\}$. Since $n',n'' \geq 2$, we can make the changes of variables $\theta' \mapsto -\theta'$ and $\theta'' \mapsto -\theta''$ to obtain \eqref{normal}. 
\end{proof}
\begin{remark}
Theorem~\ref{th:normal} makes sense when (only) one of $n'$ or $n''$ equals $1$ as well with the integral over $\sb^0=\{-1,1\}$ interpreted as a sum, see \cite{SU-book}.
\end{remark}
 
 We can extend $L$ to the space of compactly supported distributions $\mathcal{E}'(\rb^n)$ by duality, as usual. 

\subsection{The Schwartz kernel of the normal operator}
Let $\phi \in C_0^\infty(\rb_+)$. Let us define $\phi_\epsilon (t) = \epsilon^{-1} \phi(t/\epsilon)$, $\epsilon>0$. Then $\phi_\epsilon\to\delta$ in $\mathcal{S}'(\rb)$, as $\epsilon\to 0$.
For $f\in \mathcal{S}(\Rn)$, 
\begin{align*}
    L'Lf(x) &= 2 \lim_{\epsilon \to 0} \int_{\sb^{n'-1}} \int_{\sb^{n''-1}} \int_{0}^\infty \int_{\rb}  \phi_\epsilon(s) f ( x+((t+s)\theta',t\theta'' )) \d s \d  t \d S(\theta'') \d S(\theta').
\end{align*}
In the $(t,\theta'')$ integral, make the change of variables $y''=t\theta''$ to get
\begin{align*}
    L'Lf(x) &= 2 \lim_{\epsilon \to 0} \int_{\rb^{n''}} \int_{\sb^{n'-1}} \int_{\rb} \phi_\epsilon(s) \frac{f(x'+(|y''|+s)\theta', x''+y'')}{|y''|^{n''-1}} \d s \d S(\theta') \d y'' \\
    &= 2 \lim_{\epsilon \to 0} \int_{\rb^{n''}} \int_{\sb^{n'-1}} \int_{\rb} \phi_\epsilon(s-|y''|) \frac{f(x'+s\theta', x''+y'')}{|y''|^{n''-1}} \d s \d S(\theta') \d y'' \\
    &= 2 \lim_{\epsilon \to 0} \int_{\rb^{n''}} \int_{\rb^{n'} } \phi_\epsilon(|y'|-|y''|) \frac{f\left( x'+y',x''+y''\right)}{|y'|^{n'-1}|y''|^{n''-1}} \d y' \d y'',
\end{align*}
where in the  last equality, we switched to polar coordinates in the first variable. Next,
\begin{align}
    L'Lf(x) &= 2 \lim_{\epsilon \to 0}  \int_{\rb^{n''}} \int_{\rb^{n'} } \phi_\epsilon \left(|y'-x'|-|y''-x''|\right) \frac{f\left(y\right)}{|y'-x'|^{n'-1}|y''-x''|^{n''-1}} \d y' \d y'' \\
    &= 2 \lim_{\epsilon \to 0}  f \ast \frac{\phi_\epsilon(|\cdot'|-|\cdot''|)}{|\cdot '|^{n'-1} |\cdot ''|^{n''-1}}. \label{lim}
\end{align}

We want to interpret the limit above as a delta function. Set $F(x) = |x'|-|x''|$. Then $\nabla F = (x'/|x'|, -x''/|x''|)$. Therefore, $F$ is smooth, with $|\nabla F|^2=2$, in some neighborhood of $\{F=0, \, x\not=0\}$ not containing the origin. Then away from $x=0$, we have $\lim_{\epsilon\to0} \phi_\epsilon(|x'|-|x''|)=\delta(|x'|- |x''|)$, which, in particular, is positively homogeneous w.r.t.\ $x$ of order $-1$. Then the convolution kernel in \eqref{lim} is positively homogeneous of order $-n+1$ (away from $x=0$). It has a unique extension to $\Rn$ as such a distribution, see \cite[Theorem 3.2.3]{Hormander1}. On the other hand, 
 \eqref{normal} implies  that the Schwartz kernel of $L'L$ is positively homogeneous of order $-n+1$, therefore it equals that extension. Indeed, given a function $h$, denote temporarily $h_\lambda=h(\lambda x)$, $\lambda >0$. Then by \eqref{normal}, $L'Lf_\lambda = \lambda^{-1}(L'Lf)_{\lambda} $, $\forall \lambda>0$. This implies that the convolution kernel $K$ of $L'L$ satisfies $K_\lambda = \lambda^{1-n} K$, as claimed.  With that understanding, we replace the limit above by a delta (with the understanding that we made sense of the whole expression below, not just of its numerator) to get the following.

\begin{theorem}\label{thm_L''L}
	$L'L$ is a convolution with
	\[
	2\frac{\delta(|x'|-|x''|)}{|x '|^{n'-1} |x ''|^{n''-1}}.
	\]
\end{theorem} 

The denominator above can be replaced by $|x'|^{n-2}$ or by $|x''|^{n-2}$ .

\subsection{The normal operator as a Fourier multiplier} 
Since $L'L$ is a convolution, it is a  Fourier multiplier with $p(\xi)$ equal to the Fourier transform of its convolution kernel. In particular, $p$ is positively homogeneous of degree $-1$. We will compute $p$ in the following way. 
By \eqref{L'Lf}, for $f\in C_0^\infty(\rb^n)$,  
\begin{align*}
    \widehat{L'L f}(\xi) &= \int_{\rb^{n}} e^{-\I x \cdot \xi} (L'L f)(x) \d x \\
    &=  \int_{\rb^{n}} e^{-\I x \cdot \xi} \int_{\sb^{n'-1}} \int_{\sb^{n''-1}} \int_{\rb} f\left ( x + t \theta \right) \d  t \d S(\theta'') \d S(\theta') \d x.
\end{align*}
The multiple integral is not absolutely convergent in general. In order to make it such, we introduce the factor $e^{-\epsilon t^2/2}$ and eventually we will take the limit $\epsilon\to 0+$. Then the limit of the modified integral would equal the integral above, both considered as distributions in the $\xi$ variable (as we will see below, we always get a locally $L^1$ function of $\xi$ but when $n'=n''=2$, it is not bounded).  Indeed, $\langle \widehat{L'L f},\psi\rangle = \langle L'Lf,\hat\psi\rangle $ for every $\psi\in\mathcal{S}(\rb^n)$, with the pairing on the right represented by an absolute convergent integral. Then regularizing the Schwartz kernel of $L'L$ as above creates a family convergent  to $\langle L'Lf,\hat\psi\rangle $ by the dominated convergence theorem.   
  With the change of variable $x \mapsto x-t\theta$, we obtain
\begin{align*}
    \widehat{L'L f}(\xi) &= \lim_{\epsilon\to 0+} \int_{\sb^{n'-1}} \int_{\sb^{n''-1}} \int_{\rb} \int_{\rb^{n}} e^{-\I x \cdot \xi + \I  t\theta \cdot \xi -\epsilon t^2/2} f(x) \d x \d t \d S(\theta'') \d S(\theta') \\
    &=  \widehat{f}(\xi)\lim_{\epsilon\to 0+}  \int_{\sb^{n'-1}} \int_{\sb^{n''-1}} \int_{\rb} e^{ \I  t\theta \cdot \xi -\epsilon t^2/2 } \d t \d S(\theta'') \d S(\theta')\\
    &=  \widehat{f}(\xi)\lim_{\epsilon\to 0+}  \int_{\sb^{n'-1}} \int_{\sb^{n''-1}}  
    (2\pi/\epsilon)^{1/2} e^{ - (\xi \cdot\theta)^2/(2\epsilon) }   \d S(\theta'') \d S(\theta')\\
    & = \hat f(\xi) p(\xi),
\end{align*}
with the last identity defining $p(\xi)$ as a distribution. The limit is a delta at $\xi\cdot\theta$ multiplied by $2\pi$, thus the multiplier that we are looking for is given by
\begin{equation}\label{p1}
    p(\xi ) = 2\pi \int_{\sb^{n'-1}} \int_{\sb^{n''-1}} \delta(\xi' \cdot \theta' + \xi'' \cdot \theta'') \d S(\theta') \d S(\theta'').
\end{equation}
The function $\psi:=  \xi\cdot\theta$ has differential 
\[
\d_\theta\psi = (\xi'-(\xi'\cdot\theta') \theta', \xi''-(\xi''\cdot\theta'') \theta'')
\]
on  $\sb_{\theta'}^{n'-1}\times \sb_{\theta''}^{n''-1}$   vanishing on the null set when $\xi' =\lambda \theta'$, $\xi''= -\lambda \theta''$ with some $\lambda$ (and then  $|\xi'|=|\xi''|$). Thus $\delta(\psi) = \delta(\xi\cdot\theta)$ is a distribution with respect to $\theta\in\sb_{\theta'}^{n'-1}\times \sb_{\theta''}^{n''-1} $ smoothly depending on the parameter $\xi\not\in\mathcal{L}$.  Then its action on each test function (and \eqref{p1} is such) is a smooth function of $\xi\not\in\mathcal{L}$ as well. This follows, for  example, by regarding $\delta(\xi\cdot\theta)$ as the pullback of $\delta$ under the map $\theta\mapsto s=\xi\cdot\theta\in \rb$, see \cite[Theorem~7.2.1]{Friedlander1998}.

The delta above is not correctly defined as a function of $\theta$ as a pullback for $\xi$ on the light cone but the meaning of \eqref{p1}, as already mentioned, is
\begin{equation}   \label{pphi}
\langle p , \phi \rangle = 2\pi \int_{\sb^{n'-1}} \int_{\sb^{n''-1}}\int \delta(\xi \cdot \theta)\phi(\xi) \d\xi \d S(\theta') \d S(\theta''),\quad \forall \phi\in C_0^\infty(\rb^n).
\end{equation} 
The inner integral above is in distribution sense (a Radon transform), defining a smooth function of $\theta\in \sb ^{n'-1}\times \sb ^{n''-1}$.

Formally applying the Funk-Hecke theorem 
\[
\int_{\sb^{n-1}} \phi(x\cdot\theta) \d S(\theta) = |\sb^{n-2}|.|x|^{2-n} \int_{\rb} (|x|^2-t^2)_+^{(n-3)/2} \phi(t)\d t, \quad n\ge2,
\]
 we find 
    \begin{align} 
    p(\xi) &= C_{n',n''} |\xi'|^{2-n'} |\xi''|^{2-n''} \int_{\rb} \int_{\rb} (|\xi'|^2 - u^2)_+^{(n'-3)/2} (|\xi''|^2-v^2)_+^{(n''-3)/2} \delta(u+v) \d u \d v \\
    &= C_{n',n''} |\xi'|^{2-n'} |\xi''|^{2-n''} \int_{\rb} (|\xi'|^2 - u^2)_+^{(n'-3)/2} (|\xi''|^2-u^2)_+^{(n''-3)/2}  \d u ,   \label{eq:p}
\end{align}
where $C_{n',n''} =2\pi |\sb^{n'-2}|.|\sb^{n''-2}|$. Equation \eqref{eq:p} actually holds in distribution sense as it can be seen by applying the Funk-Hecke theorem to the action of $p$ on test functions. It defines a function positively homogeneous of order $-1$ in $\xi$ as we established earlier,  and \eqref{eq:p} holds in the way written for $\xi'\not=0$, $\xi''\not=0$. The cases of one of them only being zero can be handled as a limit. 

We have, using the evenness of the integrand in \eqref{eq:p}, 
\begin{align}
	p(\xi) &= C_{n',n''} |\xi'|^{2-n'}  |\xi''|^{2-n''} \int_{\rb} (|\xi'|^2 - u^2)_+^{(n'-3)/2} (|\xi''|^2 - u^2)_+^{(n''-3)/2} \d u \\
	&= 2C_{n',n''}|\xi'|^{2-n'} |\xi''|^{2-n''} \int_{0}^{\infty} (|\xi'|^2 - u^2)_+^{(n'-3)/2} (|\xi''|^2 - u^2)_+^{(n''-3)/2} \d u.
\end{align}
Set 
\begin{equation}   \label{kappa}
	\kappa = |\xi'|/|\xi''|.
\end{equation}
Then $\kappa-1$ (or $\kappa^2-1$) is a defining function of the light cone $\mathcal L$.  We use the notation $p=p_{n',n''}$ below. With the substitution $u = |\xi'|s$, we have  
\begin{equation}   \label{pp}
p_{n',n''}(\xi) = 2C_{n',n''} |\xi''|^{-1} \int_0^\infty (1-s^2)_+^{(n'-3)/2} (1-\kappa^2 s^2)_+^{(n''-3)/2} \d s.
\end{equation}

\subsection{Examples}\label{example} We consider several special cases of $n'$ and $n''$. One of our goals is to get an idea of the singularity across the light cone $\mathcal L$. We investigate these singularities in section~\ref{sec_co} from the point of view of conormal singularities. 
\begin{example}\label{Ex34}
Consider $n'=n''=3$ first. Then 
\[
p_{3,3}(\xi) = 16\pi^3 \frac{\min(|\xi'|,|\xi''|)}{ |\xi'|.|\xi''| } = \frac{16\pi^3 }{ \max(|\xi'|,|\xi''|) }
\]
since in this case, each factor in the integral  \eqref{eq:p} is a Heaviside function.   
Since $ |\xi|/\sqrt2\le \max(|\xi'|,|\xi''|) \le |\xi|$, we have
\[
\frac{16  \pi^3}{|\xi|}     \le p_{3,3}(\xi) \le \frac{16\sqrt 2 \pi^3}{|\xi|}.
\]
In particular, this shows that the singularity at $\xi=0$ is locally integrable,  thus contributing to a smoothing operator. Next, $p$ is continuous but not smooth, even after cutting away a neighborhood of $\xi=0$; therefore not a symbol of a pseudo-differential operator. However, $p$ is smooth away from the  light cone $\mathcal L$. 
Its behavior as $|\xi|\to\infty$  (but not of its $\xi$-derivatives) is like $\sim |\xi|^{-1}$. In particular, one gets
\begin{equation}   \label{est}
\|f\|_{L^2(\Omega)}/C \le \|L'Lf\|_{H^1(\rb^n)} \le C \|f\|_{L^2(\Omega)},\quad \forall f\in L^2(\Omega),
\end{equation}
see  \eqref{st2} since $|\xi| p_{3,3}(\xi)$ has positive lower and upper bounds. Here, we view $L^2(\Omega)$ as a subspace of $L^2(\rb^n)$ by using an extension as zero. The proof is the same as in \cite{SU-book}.  
\end{example}

\begin{example}
We will investigate  $p_{n',n''}$ when one of the dimensions is $3$, say $n''=3$. Note that $p_{n',n''}(\xi', \xi'') = p_{n'',n'}(\xi'',\xi')$ by \eqref{eq:p}.   
Let us assume first that $\kappa \in (0,1)$. As before, we obtain
\begin{align}
p_{n',3}(\xi) &= 2 C_{n',3} |\xi''|^{-1} \int_0^1 (1-s^2)^{(n'-3)/2} \d s \\
&= 2 C_{n',3} |\xi''|^{-1} \int_{0}^{\pi/2} \cos^{n'-2} \theta \d \theta = \frac{\sqrt{\pi} \Gamma((n'-1)/2)}{\Gamma(n'/2)}  C_{n',3} |\xi''|^{-1}, \label{n'3}
\end{align}
where $\Gamma$ is the Gamma function. Now, let us consider $\kappa \in (1,\infty)$. In this case, setting $s=\sin\theta$, we find
\begin{align}\label{pn3}
    p_{n',3}(\xi) &= 2 C_{n',3} |\xi''|^{-1} \int_{0}^{1/\kappa} (1-s^2)^{(n'-3)/2} \d s \\
    &= 2 C_{n',3}  |\xi''|^{-1} \int_{0}^{\arcsin(1/\kappa)} \cos^{n'-2} \theta \d \theta, \quad \kappa>1. 
\end{align}
Comparing to \eqref{n'3}, we see that $p_{n',3}$ is continuous across the light cone given by $\kappa=1$. It is not smooth, however. Indeed, we have 
\begin{equation} 
p_{n',3}(\xi) =  2 C_{n',3} |\xi''|^{-1} \begin{cases}
  G(\pi/2) & \text{if $\kappa\le 1$},\\
 G(\arcsin(1/\kappa))& \text{if $\kappa>1$},
\end{cases}
\end{equation}
where $G(s):= \int_0^s \cos^{n'-2} \theta \d \theta$. The function $\arcsin(1/\kappa)$ can be regularized near $\kappa=1$ with the substitution $\kappa=1+t^2$ with $t$ near $0$. We have that  $t\mapsto \arcsin(1/(1+t^2))$ is smooth for $t\ge0$ up to $t=0$, and has a non-trivial Taylor expansion at $t=0$ valid for $t\ge0$. In fact, its derivative is $-2(1+t^2)^{-1} (2+t^2)^{-1/2}$ which has even an analytic extension to $t<0$. Then $G(\arcsin(1/\kappa))$ has a non-trivial (meaning non-constant) expansion in terms of $\sqrt{\kappa-1}$ near $\kappa=1$. This shows that $p_{n',3}$ is smooth all the way to the light cone  $\mathcal L$ on both sides of it but not smooth across it. In particular, it has a conormal singularity there.

The dependence on the parity of $n'$ becomes clear by looking at \eqref{pn3}. Evaluating the last integral for particular values of $n'$, we obtain the following. 

When $n'=2$, $G(s)=s$, and 
\[
p_{2,3}(\xi) =   2 C_{2,3} |\xi''|^{-1} 
\begin{cases}
 \pi/2 & \text{if $\kappa\le 1$},\\
\arcsin(1/\kappa)& \text{if $\kappa>1$},
\end{cases}
\]
with $C_{2,3}=8\pi^2$.  
This symbol is continuous at $\kappa=1$ but the derivative from $\kappa>1$ has a singularity of the kind $(1-\kappa^2)^{-1/2}$.

When $n'=3$, $G(s)=\sin s$, and we get
\[
p_{3,3}(\xi) = 2 C_{3,3} \max(|\xi'|, |\xi''|)^{-1} 
\]
with $C_{3,3}=(2\pi)^3$, as in Example~\ref{Ex34}.  

When $n'=4$, $G(s)=\frac14 \sin(2s)+s/2$. Then  
\[
p_{4,3}(\xi) = 2 C_{4,3} |\xi''|^{-1} 
\begin{cases}
 \pi/4 & \text{if $\kappa\le 1$},\\
\frac{1}{2} \frac{\sqrt{\kappa^2 -1}}{\kappa^2} + \frac{1}{2} \arcsin{\frac{1}{\kappa}}& \text{if $\kappa>1$},
\end{cases} 
\]
where $C_{4,3}=(2\pi) (2\pi)(4\pi)$.

When $n'=5$, 
\[
p_{5,3}(\xi) = 2 C_{5,3} |\xi''|^{-1} 
\begin{cases}
 2/3 & \text{if $\kappa\le 1$},\\
\left( \frac{1}{\kappa} - \frac{1}{3\kappa^3} \right)& \text{if $\kappa>1$}.
\end{cases}
\]
One can show that the square roots (which are a part of the expansion of $\arcsin(1/\kappa)$ as well) appear for $n'$ even only. This also follows from the analysis in section~\ref{sec_co}. 
\end{example}

\section{Properties of $p(\xi)$. Inversion and stability} \label{sec_4}
\subsection{Properties of $p(\xi)$} 

We recall the definition of the hypergeometric function 
\begin{equation}   \label{hyper}
{}_2F_1(a,b;c;z) = \sum_{k=0}^\infty \frac{(a)_k(b)_k}{(c)_k} \frac{z^k}{k!},
\end{equation}
where $(a)_k:= a(a+1)\dots (a+k-1)$ when $k\ge1$, and $(a)_0=1$. We collect some properties in the following  {see \cite[Theorems~2.1.3,~2.2.2]{Andrews-SF}}. 

\begin{proposition}\label{pr_H}\cite{Andrews-SF}
	The series \eqref{hyper} converges absolutely when $|z|<1$ and $c\not\in \{0,-1,-2,\dots\}$. It	has the following behavior  as $z\to 1-$.
	\begin{itemize}
		\item[(a)] If $\Re(c-a-b) > 0$, then 
		\[
		{}_2F_1(a,b;c;1) = \frac{\Gamma(c) \Gamma(c-a-b)}{\Gamma(c-a) \Gamma(c-b)}.
		\]
		In particular, ${}_2F_1(a,b;c;z)$ is it continuous up to $z=1$ from the left by Abel's theorem.  
		\item[(b)] If $c=a+b$, then 
		\[
		\lim_{z \to 1-} \frac{{}_2F_1(a,b;c;z)}{-\log(1-z)} = \frac{\Gamma(a+b)}{\Gamma(a) \Gamma(b)}.
		\]
		\item[(c)] If $\Re(c-a-b) = 0$ and $c \neq a+b$, then 
		\[
		\lim_{z \to 1-} (1-z)^{a+b-c} \left( {}_2F_1(a,b;c;z) - \frac{\Gamma(c) \Gamma(c-a-b)}{\Gamma(c-a) \Gamma(c-b)} \right) = \frac{\Gamma(c) \Gamma(a+b-c)}{\Gamma(a) \Gamma(b)}.
		\]
		\item[(d)] If $\Re(c-a-b) < 0$, then
		\[
		\lim_{z \to 1-} \frac{{}_2F_1(a,b;c;z)}{(1-z)^{c-a-b}} = \frac{\Gamma(c) \Gamma(a+b-c)}{\Gamma(a) \Gamma(b)}.
		\]
	\end{itemize}
\end{proposition}

We recall the definition of the Beta function as well:
\[
\mathcal{B}(z_1,z_2) = \frac{\Gamma(z_1) \Gamma(z_2)}{\Gamma(z_1+z_2)}.
\]

\begin{theorem}   \label{thm_p}
	Let $n', n'' \geq 2$. $L'L$ is a Fourier multiplier with $p(\xi)$ given by \eqref{p1}, which also can be expressed as the integral \eqref{eq:p}. 

	(a) We have $p\in L^1_{\rm loc}(\rb^n)$. It is smooth away from the light cone $\mathcal L$ and away from the origin, and it is positively homogeneous of order $-1$. 
	
	(b) If $n'\ge3$ or $n''\ge3$, $p$ is continuous across the light cone $\mathcal L$ (away from the origin), and 
	\begin{equation}   \label{estC}
\frac{1}{C|\xi|}\le	p(\xi)\le \frac{C}{|\xi|}
	\end{equation}
with some $C=C(n',n'')>0$.

(c) When $n'=n''=2$, \eqref{estC} holds outside any conic neighborhood of the light cone, i.e., for $| |\xi'|-|\xi''||\ge c_0|\xi|$ with any $c_0>0$. Inside that set, for $0<c_0\ll1$, we have
\begin{equation}   \label{est_C}
\frac{1}{C|\xi|}\le	\frac{p_{2,2}(\xi)}{-\log \frac{| |\xi'|-|\xi''||}{|\xi|} }\le \frac{C}{|\xi|}.
\end{equation}
\end{theorem}

\begin{proof} 
Consider (a) first. 
The smoothness away from $\mathcal{L}\cup \{0\}$ was proved in the analysis of \eqref{p1} above.  On the other hand, since the convolution kernel of $L'L$ is positively homogeneous of order $-n+1$, $p$ must be positively homogeneous of order $-1$. The regularity claim follows from (b) and (c) below (which do not use (a)). Indeed, they imply that $p(\xi)$ is integrable on the unit sphere. Since such a homogeneous distribution $p(\xi)=|\xi|^{-1} p(\xi/|\xi|)$, of order $-1$,  has a unique extension to the whole $\rb^n$, that extension coincides with its restriction to $\xi\not=0$ considered as an almost everywhere defined function. It is in $L^1_{\rm loc}(\rb^n)$ since $n\ge4$. 

Consider (b) now. 
Since $p\in L^1_{\rm loc}$ is smooth away from the light cone, it is enough to consider the behavior of $p$ as $\xi$ tends to the latter. The integral in \eqref{eq:p} is absolutely convergent and defines a locally bounded (and, as we prove below, a continuous) function of $\xi= (\xi',\xi'')\not=0$ near the light cone (and hence everywhere away from the origin). Also, $p(\xi)>0$ for $\xi'\not=0$ and $\xi''\not=0$ but representation \eqref{eq:p} is not convenient to establish positivity when one of those factors vanishes. In that case, say when  $\xi'=0$, but $\xi'' \neq 0$, we use \eqref{p1} to write 
\begin{align*}
	p(0,\xi'') &= 2 \pi |\sb^{n'-1}| \int_{\sb^{n''-1}} \delta(\theta'' \cdot \xi'') \d S(\theta'') \\
	&= 2 \pi |\sb^{n'-1}| |\sb^{n''-2}| \int_{\rb} \delta(|\xi''| t) (1-t^2)_+^{(n''-3)/2} \d t \\
	&= 2\pi |\sb^{n'-1}| |\sb^{n''-2}| \frac{1}{|\xi''|},
\end{align*}
where we invoke Funk-Hecke theorem to get the second equality.
This proves estimate \eqref{estC} in case (b).

To prove the continuity of $p$ away from the origin in case (b), we can apply the Lebesgue dominated convergence theorem to \eqref{pp}. We will do something different however, which is of its own interest: we  derive an ``explicit'' formula for $p$ using hypogeometric functions. That allows us to treat case (c), as well.

For now, we consider the general case $n'\ge2$ and $n''\ge2$. 
Recall that $\kappa = |\xi'|/|\xi''|$ 
 when $\xi''\not=0$, see \eqref{kappa}. 
 Recall \eqref{pp}:
\begin{align}
	p(\xi) = 2C_{n',n''} |\xi''|^{-1} \int_0^\infty (1-s^2)_+^{(n'-3)/2} (1-\kappa^2 s^2)_+^{(n''-3)/2} \d s.
\end{align}
The substitution $t = s^2$ further yields
\begin{equation} \label{ppp}
	p(\xi) = C_{n',n''}  |\xi''|^{-1} \int_0^\infty t^{-1/2} (1-t)_+^{(n'-3)/2} (1-\kappa^2 t)_+^{(n''-3)/2} \d t.
\end{equation}	
When $\kappa\in (0,1)$, we get
\[
p(\xi)	=  C_{n',n''}  |\xi''|^{-1} \mathcal{B} \Big( \frac{1}{2}, \frac{n'-1}{2} \Big)\, {}_2F_1\Big(\frac{3-n''}{2}, \frac{1}{2}; \frac{n'}{2}; \kappa^2\Big),\quad \kappa\in (0,1),
\]
where in the last line, we used the Euler's Integral Representation of the Hypergeometric function (see \cite[Theorem~2.2.1]{Andrews-SF}).

Since for the coefficients $a,b,c$, see Proposition~\ref{pr_H}, we have $c-a-b= (n'+n''-4)/2\ge  1/2$, we can take the limit $\kappa\to 1-$ to get
\begin{align*}
	\lim_{\kappa \to 1-} p(\xi) &= C_{n',n''}  |\xi''|^{-1} \mathcal{B}\Big(\frac{1}{2}, \frac{n'-1}{2}\Big) \frac{\Gamma(n'/2) \Gamma((n'+n''-4)/2)}{\Gamma((n'+n''-3)/2) \Gamma((n'-1)/2)}\\
	&= C_{n',n''}  |\xi''|^{-1}   \frac{ \Gamma(1/2)  \Gamma((n-4)/2)}{\Gamma((n-3)/2) }\\
	& =C_{n',n''}   |\xi''|^{-1}   \mathcal{B}\Big(\frac{1}{2}, \frac{n -4}{2}\Big) . 
\end{align*}
This expression is symmetric with respect to $|\xi'|$ and $|\xi''|$ (note that we can replace $|\xi''|$ and $|\xi'|$ when $\kappa=1$), therefore, the limit $\kappa\to 1+$ would yield the same result. Therefore, $p$ is continuous across the light cone away from the origin. 

We proceed with the proof of (c) now. We have $n'=n''=2$. 
For $\kappa \in (0,1)$, we have 
\begin{align*}
	p_{2,2}(\xi) &= 8 \pi^2 |\xi''|^{-1} \,{}_2F_1\Big(\frac{1}{2}, \frac{1}{2}; 1; \kappa^2\Big),
\end{align*}
which is also $16\pi |\xi''|^{-1} $ times the complete elliptic integral of the first kind. 
Therefore, by Proposition~\ref{pr_H}(b), 
\begin{align*}
	\lim_{\kappa \to 1-} \frac{p_{2,2}(\xi)}{-\log(1-\kappa^2)} = 8 \pi |\xi''|^{-1}.
\end{align*}
For $\kappa \in (1,\infty)$, the expression for $p$ is
\[
p_{2,2}(\xi) = 8 \pi^2 |\xi'|^{-1} {}_2F_1\Big(\frac{1}{2}, \frac{1}{2}; 1; \kappa^{-2}\Big).
\]
As $\kappa \to 1+$, $\kappa^{-2} \to 1-$ and we get by Proposition~\ref{pr_H}(b) again,
\[
\lim_{\kappa \to 1+} \frac{p_{2,2}(\xi)}{-\log(1-\kappa^{-2})} = 8 \pi |\xi'|^{-1}.
\]
Note that $\log(1-\kappa^{-2})= \log\kappa^{-2} + \log|1-\kappa^2|= \log|1-\kappa^2|+o(1) = \log{|1-\kappa|} + o(1)$, as $\kappa\to1$. Also, $\log|1-\kappa| = \log\frac{||\xi'|-|\xi''| |}{|\xi''|} = \log\frac{||\xi'|-|\xi''| |}{|\xi|}+ O(1)$ for $\kappa$ close to $1$. Therefore,
\begin{equation}   \label{limp}
\lim_{\kappa \to 1} \frac{|\xi| p_{2,2}(\xi)}{-\log\frac{||\xi'|-|\xi''| |}{|\xi|}} = 8 \pi .
\end{equation}
This implies  estimate \eqref{est_C}.
\end{proof}

\begin{remark}
	In case (b), $p(\xi)$ is continuous across the light cone but it is not smooth there, by \cite[Theorem~7.1.18]{Hormander1}. We study this in more detail below. 
\end{remark}

\subsection{Inversion and stability} 

\begin{corollary}[Inversion]
	Let $n',n'' \geq 2$. For every $f\in L^2(\Omega)$, we have
	\[
	f = p(D)^{-1}L'Lf,
	\]
    where the operator $p(D)$ on the right is defined via the Fourier transform, as follows
    \[
p(D)^{-1} L'L f \coloneqq \mathcal{F}^{-1} p^{-1} \mathcal{F} L'L f .
    \]
\end{corollary}

The proof is straightforward in case (b) of Theorem~\ref{thm_p}. Indeed, then $\widehat{L'Lf}= p \hat f$ with $\hat f$ smooth and $p\in L^1_\textrm{\rm loc}$ satisfying \eqref{estC}. Then we just divide by $p$.

In case (c), when $n'=n''=2$, the symbol $p_{2,2}$ is singular at $\mathcal{L}$, and $p_{2,2}^{-1}$ extends continuously as zero there. For $f\in L^2(\Omega)$, $\hat f$ is smooth, and then $p_{2,2}\hat f$ would have the same singularity, at worst, as $p_{2,2}$ does, at $\mathcal L$. Then $ p_{2,2}^{-1}(\xi) p_{2,2}(\xi)\hat f(\xi)$ equals $\hat f(\xi)$ away from $\mathcal{L}$. Therefore, they equal almost everywhere, which is enough to determine $\hat f$ as an element of $L^2$. We are essentially using the fact that $\supp \hat f$ cannot be in $\mathcal L$. 

To formulate a stability estimate, in the case $n'=n''=2$ we need a modified Sobolev space. To this end, define the norm
\[
 \|h\|_{H_{2,2}^1(\rb^4)}^2:= \int_{\rb^4}\sigma^2(\xi)|\hat h(\xi)|^2\d\xi, \quad \sigma(\xi):=   \frac{\langle \xi\rangle }{-\log\frac{||\xi'|-|\xi''| |}{e|\xi|}},
\]
where $\langle \xi \rangle = (1+|\xi|^2)^{1/2}$. Since $||\xi'|-|\xi''| |\le |\xi|$, the denominator in $\sigma$ has a lower bound $1$, therefore $0\le\sigma(\xi)\le \langle \xi\rangle$ but $\sigma$ is not elliptic. We define the space $H_{2,2}^1(\rb^4)$ as the completion of $C_0^\infty(\rb^4)$ under that norm. It is a Hilbert space. 

\begin{theorem} 
	\ 
	
	(a) When $n'=n''=2$, we have 
	\begin{equation}   \label{st1}
	\|f\|_{L^2(\Omega)}/C \le \|L'Lf\|_{H^1_{2,2}(\rb^4)} \le C \|f\|_{L^2(\Omega)},\quad \forall f\in L^2(\Omega). 
	\end{equation}
	
	(b)  When either $n'\ge3$ or $n''\ge3$, we have
	\begin{equation}   \label{st2}
		\|f\|_{L^2(\Omega)}/C \le \|L'Lf\|_{H^1(\rb^n)} \le C \|f\|_{L^2(\Omega)},\quad \forall f\in L^2(\Omega).
	\end{equation}
\end{theorem}

\begin{proof}
We start with \eqref{st1}. 	We take Fourier transforms of $f$ and $L'Lf$. Then the first inequality follows from $p(\xi)\sigma(\xi)\ge 1/C$, which is true in a conic neighborhood of the light cone $\mathcal L$ by \eqref{limp}, and outside it in a trivial way. For the second inequality in \eqref{st1},  we can use   $p(\xi)\sigma(\xi)\le C$ for $|\xi|\ge1$. In the ball $|\xi|\le 1$ however, $p$ has a singularity at the origin, while $\sigma$ does not compensate for it; it compensates for the logarithmic factor only. To deal with the singularity at $\xi=0$, we proceed as in the proof of \cite[Theorem~II.5.1]{Natterer-book}. Denote by $q(\xi)$ the function in \eqref{limp}, before we take the limit there. It is positively homogeneous of order $0$, and has positive lower and upper bounds. We have $p\sigma = q \langle\xi\rangle/|\xi|$, and then $1/C \le |\xi|p\sigma\le C$. Then cutting $\widehat{L'Lf}$ to $|\xi|\le1$, we are left with estimating
\[
\int_{|\xi|\le 1} |\xi|^{-2} |\hat f(\xi)|^2\d \xi
\]
which can be bounded by $C_s\|f\|^2_{H^s}$ for every $s$, see p.~44 of the proof of  \cite[Theorem~II.5.1]{Natterer-book}   as a consequence of the fact that $|\xi|^{-2}$ is locally integrable, and that $f$ is compactly supported. 

The proof of (b) is similar, except that we have $\sigma(\xi) = \langle \xi\rangle$ then. 
\end{proof}

\subsection{The symbol $p(\xi)$ is a conormal distribution for $\xi\not=0$} \label{sec_co}
We established already that $p(\xi)$ is singular only at the light cone $\mathcal L$. 
We will show that away from the origin, the singularities are conormal at the light cone. We refer to  \cite[sec.~18.2]{Hormander3} for the conormal distributions calculus.   Our starting representation is \eqref{ppp}. One can see even directly from it that the only singularity with respect to $\kappa\in (0,\infty)$   happens at $\kappa=1$. Then only the integral in some neighborhood of $t=1$ matters for the singularity. With this in mind, we set $s=1-t$, $z = \kappa^{2}-1$ in \eqref{ppp}. In the new variables, we study the singularity at $z=0$; and it will be affected by the behavior of the integrand near $s=0$ only.    Then $1-\kappa^{2} t = 1- (1+z)(1-s) = s-z +sz $, and 
\begin{equation}\label{q}
p(\xi) = C_{n',n''} |\xi''|^{-1}q,\quad q(z) := \int_{\rb} (1-s)_+^{-1/2} s_+^{(n'-3)/2} (s-z +sz)_+^{(n''-3)/2} \d s .
\end{equation}
We have $z=|\xi'|^2/|\xi''|^2-1$; in other words, eventually, we will view $q$ as $q=q(z(\xi))$. 

Recall that 
\[
\big(x_+^{\lambda-1}\big)\hat{\big.} = \Gamma(\lambda)  e^{-\I \pi \lambda/2} (\xi-\I 0)^{-\lambda}  , \quad \lambda \in\mathbb{C}\setminus\{0,-1,\dots\},
\]
where the distributions $x_\pm^{\alpha}$ are defined as a meromorphic continuation from $\Re\alpha>-1$.  
In our case, we will apply this with $\lambda-1\in \{ -1/2,0,1/2,1,\dots \}$. The resolvent $(\xi-\I 0)^{-\lambda}$ is defined as the limit, as $\epsilon\to0+$, of $(\xi-\I\epsilon)^{-\lambda} = e^{\I \lambda \pi/2}(\I \xi+\epsilon)^{-\lambda}$, with $z^{-\lambda}$ defined in $|\arg z|<\pi/2$ with the branch satisfying $1^{-\lambda}=1$, see \cite[Example~7.1.17]{Hormander1}. 

Then we have, see \cite[p.~110]{Friedlander1998}, \cite[p.~72]{Hormander1},
\[
\big(x_+^{\lambda-1}\big)\hat{\big.} = \Gamma(\lambda) \left( e^{-\I \pi \lambda/2} \xi_+^{-\lambda} + e^{\I \pi \lambda/2} \xi_-^{-\lambda} \right), \quad \lambda \not\in\mathbb{Z},
\]
where $\xi_-=(-\xi)_+$, and
\begin{equation}\label{x1}
\big(x_+^{\lambda-1}\big)\hat{\big.} = (-\I)^\lambda (\lambda-1)! \left(  \xi_+^{-\lambda} +  (-1)^\lambda \xi_-^{-\lambda} \right) = (-\I)^\lambda (\lambda-1)! \,  \xi^{-\lambda}, \quad \xi\not=0, \; \lambda =1,2,\dots.
\end{equation}

The second relation can be obtained by starting with the Fourier transform of the Heaviside function and applying $(\I \partial_\xi)^{\lambda-1}$. For half-integers, we obtain
\begin{equation}\label{x2}
\big(x_+^{m-1/2}\big)\hat{\big.} = \Gamma(m+1/2) (-\I)^m e^{-\I\pi/4} \left(   \xi_+^{-m-1/2} + \I (-1)^m   \xi_-^{-m-1/2} \right), \quad m \in\mathbb{Z}.
\end{equation}

The symbol of $x_+^{\lambda-1}$, as a distribution conormal to $x=0$ is given by the same formulas, away from $\xi=0$. Note that this is not the same convention about the constant multiplier, a power of $2\pi$, as in \cite{Hormander3}. It would be, if we regard $x_+^{\lambda-1}$ as a conormal distributions in two dimensions; in our case with coordinates $z$ (or $\kappa$) and $|\xi''|$ as a radial variable. 

Denote the symbols of $x_+^{(k-3)/2}$ by $a_{k}(\xi)$ given either by \eqref{x1} or \eqref{x2}, depending on the parity of $k$. We assume that the origin $\xi=0$ is cut off smoothly.  
Then 
\[
(s-z +sz)_+^{(n''-3)/2}  \sim (2\pi)^{-1} \int e^{\I (s-z +sz)\zeta} a_{n''}(\zeta)\d \zeta,
\]
where $\sim$ means that the error is smooth both in $s$ and $z$. 
Next, $(1-s)_+^{-1/2} s_+^{(n'-3)/2} $ is also a  conormal distribution at $s=0$ with principal symbol $a_{n'}(\sigma)$  but a complete symbol $\tilde a_{n'}(\sigma) $ having some lower order terms compared to  $a_{n'}(\sigma)$. It has an additional singularity at $s=1$ but that does not contribute to a singularity of $	q(z)$ as explained in the first paragraph of this subsection. 
 
By \eqref{q} we get 
\begin{align}
 q(z)& \sim  (2\pi)^{-1} \iint (1-s)_+^{-1/2} s_+^{(n'-3)/2}  e^{\I (s-z +sz)\zeta} a_{n''}(\zeta)\d \zeta   \d s \\
 &=  (2\pi)^{-1}  \int e^{-\I z\zeta}a_{n''}(\zeta) \int (1-s)_+^{-1/2} s_+^{(n'-3)/2}  e^{\I s(1 +z)\zeta} \d s \d \zeta   \\
  &\sim  (2\pi)^{-1}  \int e^{\I z\zeta} a_{n''}(-\zeta) \tilde a_{n'}((1+z)\zeta) \d \zeta .  \label{qz}
\end{align}
Those calculations can be justified by considering the integrals above as oscillatory ones. 
The function $\tilde a_{n'}((1+z)\zeta)$ is a symbol as well (with $z$ close to $z=0$ a spatial variable), with principal symbol  $ a_{n'}(\zeta)$. Then we get the principal symbol (in those coordinates) of $q$ to be
\begin{equation}\label{b}
b(\zeta):=   a_{n'}(\zeta) a_{n''}(-\zeta).
\end{equation}
Its order is $-(n'-1)/2 -(n''-1)/2 = -n/2+1$. 

We proved the following. 
\begin{theorem} \label{thm_co}
The symbol $p(\xi)$, in the coordinates (parameters) $r'=|\xi'|$, $r''=|\xi''|$, $\theta'$, $\theta''$,  depends on $(r',r'')$ only, in which $\mathcal L$ is given by $r'=r''$. It has a conormal singularity at $\mathcal L$. 
 In the coordinates $z= (r'/r'')^2-1$, $r''$, 
its  principal symbol is 
\[
\sigma_0(p)= C_{n',n''}(r'')^{-1}  a_{n'}(\zeta) a_{n''}(-\zeta),
\]
where $a_{n'}$ is given either by \eqref{x1} with $\lambda -1 = (n'-3)/2$ when $n'$ is odd, or by \eqref{x2} with $m-1/2=(n'-3)/2$ when $n'$ is even.   The amplitude $a_{n''}$ is defined similarly. 
\end{theorem}

The first claim follows from \eqref{q}.   We see that $\sigma_0(p)$ is elliptic since it does not vanish.   

The invariant way of looking at a principal symbol of a conormal distribution is to view $p(\xi)$ as a half-density first. Even though it depends on $\xi$ only, we may want to think of it as a special case of a general symbol depending on $(x,\xi)$; then $p$ would be associated to $p(\xi)|\!\d x\d\xi|^{1/2}$. Then the principal symbol of it at $\mathcal{L}$, the latter considered as a submanifold of $T^*\rb^n\setminus 0$, would be a half-density on $N^*\mathcal{L}$. This would change the order by the usual convention, see \cite[XVIII.18.2]{Hormander3}. We will not pursue this. The formula for $\sigma_0(p)$ we got, multiplied by $|\!\d r''\d \zeta|^{1/2}$ gives the principal symbol of $p|\!\d z\d r''|^{1/2}$, and then should be considered to have  order $(3-n)/2$, with the extra $1/2$ coming from the half-density. 

Theorem~\ref{thm_co} suggests that $L'L$ is an FIO in the $I^{p,\ell}$ class, see \cite{Greenleaf_Uhlmann90, Greenleaf_UhlmannCM, Greenleaf-Uhlmann} mentioned earlier for application to Integral Geometry.  

\subsection{Example~\ref{example}, revisited} 
We revisit some of the particular cases for $(n',n'')$ in Example~\ref{example} to show how we can connect the explicit expressions there with the computation for the principal symbol. Recall that $b$ is given by \eqref{b}. 

\subsubsection{$(n',n'')=(3,3)$} In this case, $C_{3,3}=(2\pi)^3$, and 
\begin{align*}
	p_{3,3} &= 2C_{3,3} |\xi''|^{-1} \begin{cases}
		1, &\text{if $\kappa \leq 1$}, \\
		\frac{1}{\kappa}, & \text{if $ \kappa > 1$}
	\end{cases} 
\quad 	= \quad 2C_{3,3} |\xi''|^{-1} \begin{cases}
		1, &\text{if $z \leq 0$}, \\
		(1+z)^{-1/2}, & \text{if $z > 0$}.
	\end{cases}
\end{align*}
Thus,
\[
p_{3,3} = C_{3,3}|\xi''|^{-1}\Big(2 -z_++\dots\Big).
\]
Since $\widehat{z_+} = - \zeta^{-2}$ for $\zeta\not=0$ by \eqref{x1}, and the first term is smooth, we get $b(\zeta) = \zeta^{-2}$. 

We use Theorem~\ref{thm_co} now to compute $b$ now. Since $a_{n'}=a_{n''} = -\I \zeta^{-1}$ when $n'=n''=3$, 
\[
b(\zeta) = a_{3}(\zeta) a_{3}(-\zeta) = (-\I) \zeta^{-1} \cdot (-\I) (-\zeta)^{-1} = \zeta^{-2}, 
\]
which matches the previous computation.

\subsubsection{$(n',n'')=(5,3)$}
Recall that 
\begin{align*}
	p_{5,3}(\xi) &=  2 C_{5,3} |\xi''|^{-1} \begin{cases}
		\frac{2}{3}, & \text{if $\kappa \leq 1$} \\
		\frac{1}{\kappa} - \frac{1}{3\kappa^3}, & \text{if $\kappa > 1$}.
	\end{cases}
\end{align*}
Let us rewrite this in terms of $z = \kappa^2 - 1$:
\begin{align*}
	p_{5,3}(\xi) &=  2 C_{5,3} |\xi''|^{-1} \begin{cases}
		\frac{2}{3}, & \text{if $z \leq 0$}, \\
		(1+z)^{-1/2} - \frac{1}{3} (1+z)^{-3/2}, & \text{if $z > 0$}.
	\end{cases}
\end{align*}
Computing the first and the second  derivatives from  $z \to 0+$, or just computing the Taylor expansion, we get
\[
	p_{5,3} =  C_{5,3} |\xi''|^{-1}\Big( \frac43 -\frac12 z_+^2 +\dots \Big).
\]
The leading singularity determining $b$ is $-z_+^2/2$ having a symbol $b= (-1/2) 2\I \zeta^{-3} = - \I \zeta^{-3}$ by \eqref{x1}. 

On the other hand, with $\lambda' = (n'-3)/2 + 1  = 2$, and $\lambda'' = (n''-3)/2+1 = 1$ in \eqref{x1}, we get 
\[
a_{5}(\zeta) a_3(-\zeta) = (-\I)^2 \zeta^{-2} \cdot (-\I) (-\zeta)^{-1} = -\I \zeta^{-3} , 
\]
which coincides with what we got for $b$ above.

\subsection{$(n',n'')=(2,3)$} We have
\begin{align*}
	p_{2,3} &= 2 C_{2,3} |\xi''|^{-1} \begin{cases}
		\frac{\pi}{2}, & \text{if $\kappa \leq 1$},\\
		\arcsin{1/\kappa}, & \text{if $\kappa > 1$}.
	\end{cases}\\
	&=    2 C_{2,3} |\xi''|^{-1} \begin{cases}
		\frac{\pi}{2}, & \text{if $z \leq 0$},\\
		\arcsin{((1+z)^{-1/2})}, & \text{if $z > 0$}.
	\end{cases}
\end{align*} 
We see that $p_{2,3}$ is continuous across $z=0$.  
The first derivative of $f(z)= \arcsin{((1+z)^{-1/2})}$ is $f'=\frac{-1}{2 \sqrt{z} (1+z)}\sim -\frac12 z^{-1/2} +O(z^{1/2})$. This implies $f = \pi/2-z^{1/2}+O(z^{3/2})$. Then $b(\zeta)$ is the principal symbol of the leading singularity of $\pi-2 z_+^{1/2}$. We have 
\begin{align*}
	-2 (z_+^{1/2})^{\widehat{}} &= -2 \Gamma(3/2) (-\I) e^{-\I \pi/4} (\zeta_+^{-3/2} - \I \zeta_-^{-3/2}) \\
	&= \I \sqrt{\pi} \frac{1}{\sqrt{2}} (1-\I) (\zeta_+^{-3/2} - \I \zeta_{-}^{-3/2}), \quad \zeta\not=0.
\end{align*}

Now, we appeal to the theorem. We have $\lambda' = -1/2$, which means that $m' = 0$, and $\lambda'' = 1$. Therefore,
	\begin{align*}
	b(\zeta)= 	a_{2}(\zeta) a_{3} (-\zeta) &= \Gamma(1/2) e^{-\I \pi/4} \left(\zeta_+^{-1/2} + \I \zeta_-^{-1/2} \right) \cdot (-\I) (-\zeta)^{-1}\\
		&= \I \sqrt{\frac{\pi}{2}} (1-\I)  ( \zeta_+^{-3/2} -\I \zeta_-^{-3/2}).
	\end{align*} 
It is the same expression as above. 

\section{$L$ as a generalized Radon transform} \label{sec_FIO}
We will fit $L$ into the framework of Guillemin's generalized Radon transforms, and in particular, will describe it as a Fourier Integral Operator (FIO). We are not going to get deeper into the analysis and will only hint on why $p$ might be expected to have a non-symbol behavior (a conormal singularity at $\mathcal{L}$) from an FIO point of view. 

This framework starts with picking a connected manifold $Y$ where our functions are defined, and a connected manifold $X$ of submanifolds over which we integrate.   In our case, $Y=\rb^n$ with $n_Y=\dim Y=n$, and we identify the light lines with  $X=\Sigma$ with $n_X= \dim X=2n-3$, see \eqref{Sigma}. We choose measures on each one of them, which we already did in the previous sections. The point-line relation $Z\subset X\times Y = \Sigma\times \rb^n$  is given by 
\begin{equation}
Z:= \{(z,\theta,x)|\; (z,\theta)\in \Sigma, \, x=z+t\theta\; \text{for some $t\in\rb$}\}
\end{equation}
We parameterize $Z$ by $(z,\theta,t)\in \Sigma\times \rb$ making it a $n_Z:= 2n-2$ dimensional connected manifold, and we chose the product measure on it. 

Consider the diagram
\begin{equation} \label{diag1}
	\begin{tikzcd}[]
		&  Z \arrow{dl}[swap]{\pi_{\Sigma}}\arrow{dr}{\pi_{\rb^n}} & \\
		\Sigma  
		& &\rb^n
	\end{tikzcd} 
\end{equation}
where $\pi_\Sigma$ and $\pi_{\rb^n}$ are the natural projections. The projection $\pi_{\rb^n}$ is proper, and both are submersions. Then $Z$ is a double fibration. The transform $L$ is the generalized Radon transform related to it. Its Schwartz kernel is $\delta_Z$ with respect to the measure on $Z$. As a consequence, $L\in I^\frac{1-n}4(\Sigma\times\rb^n; \; N^*Z\setminus 0)$. Indeed, the order of $L$ is  $m= k/2- (n_X+n_Y)/4$, where $k =n_X+n_Y-n_Z $ is the codimension of the submanifolds over which we integrate on. We have $k=n-1$, and $m= (n-1)/2 - (2n-3+n)/4= (1-n)/4$.

The microlocal version of the diagram \eqref{diag1} is
\begin{equation} \label{diag2}
	\begin{tikzcd}[]
		&  N^*Z\setminus 0 \arrow{dl}[swap]{\pi_{T^*\Sigma}}\arrow{dr}{\pi_{T^*\rb^n}} & \\
		T^*\Sigma\setminus 0 
		& &T^*\rb^n\setminus 0.
	\end{tikzcd} 
\end{equation}
The dimensions in \eqref{diag2} from left to right are: $4n-6\ge 3n-3\ge  2n$. 

As we found out, $L'L$ is not a \PDO\ because $p(\xi)$ is singular at the light cone $\mathcal L$. This is an indication that the Bolker condition fails. We will demonstrate this directly. We recall that the Bolker condition requires $\pi_{T^*\Sigma}$ to be an injective immersion.

We analyze the Lagrangian $N^*Z$, and therefore, the canonical relation $N^*Z'\setminus 0$ in more detail. 
The operator $L$ has a delta type of Schwartz kernel with a wavefront set conormal to $Z$. Choose local coordinates $\alpha\in \rb^{2n-3}$ for $(z,\theta)$ near some $(z_0,\theta_0)$. Write $\ell_\alpha(t)=z+t\theta$ in these coordinates.  Then $Z$ is parameterized by $(\alpha,t)$, and
\begin{equation} \label{MA_light_2}
N^*Z = \Big\{  ((\alpha,x) , ({\hat \alpha},\xi )     )\big| \; x=\ell_\alpha(t) ,\,  \xi\cdot\dot\ell_\alpha(t)=0,\, \hat\alpha= \xi_j\partial_\alpha\ell^j_\alpha(t)  \Big\},
\end{equation}
where hats indicate dual variables, see also \cite{SU-book}. Thus, parameters on $N^*Z$ can be chosen to be $(\alpha,t,\xi)$ with $\alpha\in \rb^{2n-3}$,   $\xi\perp\theta$, $t\in\rb$. 
The canonical relation associated to $L$ is given by $C = N^*Z'\setminus 0$. 
We can check directly that  $\hat \alpha=0$ if and only if $\xi=0$.   
As a consequence, 
\be{MA_light_4}
\WF(L f) \subset C\circ \WF(f)
\ee
as it follows from the H\"ormander-Sato lemma. We have
\be{MA_light_6}
\pi_{T^*\Sigma} ((z,\theta,{\hat z},\hat\theta) , (x, \xi )) =     (z,\theta,{\hat z},\hat\theta)   = \left(z,\theta,\xi, t\big(\xi' -(\xi'\cdot\theta') \theta' , \xi'' -(\xi''\cdot\theta'') \theta''\big) \right) ,
\ee
written in the parameterization  $(z,\theta,t,\xi)$. Then injectivity of $\pi_{T^*\Sigma} $ is equivalent to unique solvability of $t\big(\xi' -(\xi'\cdot\theta') \theta' , \xi'' -(\xi''\cdot\theta'') \theta''\big)  =  \hat\theta$ with respect to $t$ given $(z,\theta,\xi)$. Uniqueness holds if and only if the vector multiplying $t$ on the left does not vanish.  When it does, uniqueness fails, which is a failure of the Bolker condition. In that case, $\xi'=\lambda\theta'$, $\xi''=\mu\theta''$, which, together with $\xi\cdot\theta=0$ implies $\xi\in\mathcal{L}$. On the other hand, $\xi\in\mathcal{L}$ implies non-uniqueness for some $\theta$.

We consider the projection
    \be{MA_light_7}
\pi_{T^*\rb^n} ((z,\theta,{\hat z},\hat\theta) , (x, \xi )) =     (x, \xi )  = (z+t\theta, \xi) 
    \ee
briefly as well. Recall that $(z,\theta,t,\xi)$ are parameters here with $z$ and $\xi$ perpendicular to $\theta\in \sb^{n'-1}\times \sb^{n''-1}$. This projection is surjective because given $\xi\in \rb^n\setminus 0$, we can always find $\theta$ as above normal to it, as we saw in the proof of Corollary~\ref{cor_uniq}. This is in contrast to the Minkowski case (either $n'=1$ or $n''=1$), where the range of $\pi_{T^*\rb^n}$ is the causal cone only since the equation $\xi\cdot(1,\theta)=0$ with $|\theta|=1$ has a solution when $|\xi'|\le |\xi''|$ only, say when $n'=1$. For this reason, we can resolve all singularities in the case $n'\ge2$, $n''\ge2$ under consideration, while we can do that only on the spacelike cone in the Minkowski case. 

We formulate the properties we established below.

\begin{theorem}\label{thm_FIO} \ 

(a) $L\in I^\frac{1-n}4(\Sigma\times\rb^n; \; N^*Z\setminus 0)$ with $\Sigma$ as in \eqref{Sigma} and canonical relation given by $C = N^*Z'\setminus 0$. 

(b) The Bolker condition is satisfied if and only if $\xi\not\in\mathcal L$.   
\end{theorem}

The theorem allows us to use pseudo-differential cutoffs either near the light cone, or near the image of it under $C$, as it is done in \cite{LOSU-Light_Ray, LOSU-strings} in the Lorentzian case to get recovery of not-lightlike singularities, say with partial measurements, or when the metric is not flat. The treatment of lightlike singularities is more delicate and requires special classes of FIOs corresponding to two cleanly intersecting Lagrangians, see \cite{Greenleaf_Uhlmann90, Greenleaf_UhlmannCM, Greenleaf-Uhlmann}, and the remark after Theorem~\ref{thm_co}. This is done in \cite{Yiran-Light-2021} for the light ray transform in Lorentzian geometry.  We do not pursue this direction in the present paper. 

\section{Comparison to the light ray transform in Minkowski spacetime}\label{minkowski}

Throughout this paper, we assumed $n'\ge2$, $n''\ge2$. The case of one of them being $1$ is the Minkowski case studied previously, as explained in the introduction. We review briefly some of the results for that case to compare them to the rest of the paper, see also \cite{S-Lorentz_analytic, LOSU-Light_Ray, SU-book}.

In the Minkowski case, we routinely call the dimension $1+d$ (not $n$ or $d$!), i.e., $(t,x)\in \rb^{1+d}$. Then $n'=1$, $n''=d$, and $n=1+d$. We also denote the dual variable by $(\tau,\xi)$.  The light ray transform is defined as
\be{MA_light_Ldef}
L f(z,\theta) = \int_\rb   f(s,z+s\theta)\,\d s, \quad (z,\theta) \in\rb^d\times \sb^{d-1}.
\ee
 Note that we do not have $(0,z)\perp(1,\theta)$ (but we could have chosen the initial point to be on $(1,\theta)^\perp$), and we integrate over future pointing light like rays, i.e., those in the direction of $(1,\theta)$, omitting $(-1,\theta)$ because that leads to the same kind of integral. This leads to some mismatch of the constants if we compare the formulas for $L'L$ in the previous sections when formally $n'=1$ there. Then $L'L$ is a convolution with
 \[
 \frac{\delta(t-|x|) +\delta(t+|x|) }{|x|^{d-1}},
 \]
compare with Theorem~\ref{thm_L''L}. We have
\[
p(\tau, \xi) =  2\pi|\sb^{d-2}|  \frac{(|\xi|^2-\tau^2)_+^\frac{d-3}2} {|\xi|^{d-2}}  .
 \]
 This replaces \eqref{eq:p} in this case. The multiplier $p$ is still positively homogeneous of degree $-1$, non-negative but it vanishes identically for $|\xi|<|\tau|$, i.e., in the timelike cone in $T^*\rb^{1+d}$. It is singular at the lightlike cone. As a result, $\hat f(\xi)$ in the timelike cone cannot be recovered stably, and timelike singularities with localized measurements are lost. This is in contrast to the case $n'\ge2$, $n''\ge2$ we study in this paper.

\providecommand{\bysame}{\leavevmode\hbox to3em{\hrulefill}\thinspace}
\providecommand{\MR}{\relax\ifhmode\unskip\space\fi MR }
\providecommand{\MRhref}[2]{%
  \href{http://www.ams.org/mathscinet-getitem?mr=#1}{#2}
}
\providecommand{\href}[2]{#2}

\end{document}